\documentclass[a4, 11pt]{amsart}
\usepackage{amsthm,amsmath,amssymb,graphics,hyperref}
\hypersetup{pdfpagemode=FullScreen,colorlinks=true}
\usepackage{graphicx}
\usepackage{psfrag}
\usepackage{amscd}
\usepackage{enumerate}

\newcommand{\ra}{{\rightarrow}}

\renewcommand{\H}{\mathbb{H}}

\newcommand{\C}{\mathbb{C}}

\DeclareMathOperator{\Hom}{Hom}

\DeclareMathOperator{\out}{Out}

\def\pslc {\mathrm{PSL}(2,\C)}

\theoremstyle{plain}
\newtheorem{thm}{Theorem}[section]

\newtheorem{lem}[thm]{Lemma}
\newtheorem{prop}[thm]{Proposition}
\newtheorem{dfn}[thm]{Definition}
\numberwithin{equation}{section}

\begin{document}
\author{Inkang Kim}
   \address{School of Mathematics,
   KIAS, Heogiro 85, Dongdaemun-gu,
   Seoul, 130-722, Republic of Korea}
   \email{inkang@kias.re.kr}


   \author{Michelle Lee}
   \address{University of Maryland,
Mathematics Building,
College Park, MD 20742, USA}
   \email{mdl@umd.edu}
\title[Separable-stable representations]
        {Separable-stable representations of a Compression body}
        \date{}
        \maketitle

\begin{abstract}
Let $M$ be a hyperbolizable, nontrivial compression body without toroidal boundary components.  In this paper, we characterize which discrete and faithful representations of $\pi_1(M)$ into $\pslc$ are separable-stable.  The set of separable-stable representations forms a domain of discontinuity for the action of $\out(\pi_1(M))$ on the $\pslc$-character variety of $\pi_1(M).$
\end{abstract}
\footnotetext[3]{The first
 author gratefully acknowledges the partial support of  NRF grant
(R01-2008-000-10052-0).  The second author gratefully acknowledges the partial support of NSF RTG grant DMS 0602191 }
\section{Introduction}
Recently, the study of the dynamics of $\out(\pi_1(M))$ on the $\pslc$-character variety of $\pi_1(M),$ for a compact hyperbolizable $3$-manifold $M$ with nonempty boundary and no toroidal boundary components, has led to the somewhat surprising discovery that often the dynamical and geometric decompositions of the $\pslc$-character variety do not coincide.  The $\pslc$-character variety of $\pi_1(M),$ denoted $\mathcal{X}(\pi_1(M), \pslc),$ is the geometric quotient of $\Hom(\pi_1(M), \pslc)$ by inner automorphisms of $\pslc.$  Sitting inside $\mathcal{X}(\pi_1(M), \pslc)$ is $AH(M),$ the set of discrete and faithful representations, which can equivalently be thought of as the deformation space of hyperbolic $3$-manifolds homotopy equivalent to $M.$  The action of $\out(\pi_1(M))$ on the interior of $AH(M)$ is properly discontinuous (see Canary \cite{can-survey}), but if $M$ has a primitive essential annulus, the action cannot be properly discontinuous on all of $AH(M)$ (Canary-Storm \cite{can-sto}).  

Minsky in \cite{Min} was the first to observe that there exists a domain of discontinuity, called the set of primitive-stable representations, for $\out(\pi_1(M))$ containing the interior of $AH(M)$ as well as points on $\partial AH(M),$ when $M$ is a handlebody.  In particular, the set of primitive-stable representations contains both discrete and faithful representations and dense representations.  Jeon-Kim-Ohshika-Lecuire \cite{JKO} gave a complete criterion for a discrete and
faithful representation to be primitive-stable in terms of disc-busting properties of the ending lamination and parabolic loci of the associated hyperbolic $3$-manifold.  In her thesis \cite{Lee}, the second author gave a generalization of the primitive-stable condition to representations of the fundamental group of a compression body into $\pslc$, called the separable-stable condition, and she showed that the set of separable-stable representations is a domain of discontinuity for the action of $\out(\pi_1(M))$ that contains the interior of $AH(M)$ as well as point on $\partial AH(M).$  In this paper, we give a characterization of which discrete and faithful representations of $\pi_1(M)$ into $\pslc$ are separable-stable analogous to the Jeon-Kim-Ohshika-Lecuire criterion.

A compression body $M$ is the boundary connect sum of a $3$-ball, a collection of trivial $I$-bundles over surfaces, and a handlebody, where the other components are connected to the $3$-ball along disjoint discs.  Its fundamental group is a free product of closed surface groups and a free group. If $M$ is not the boundary connect sum of two trivial $I$-bundles over closed surfaces, then an element in $\pi_1(M)$ is separable if it lies in a proper factor of a free decomposition of $\pi_1(M).$  
If $M$ is the boundary connect sum of $S_1 \times I$ and $S_2 \times I,$ where $S_1$ and $S_2$ are closed surfaces, then an element is separable if it lies in a factor of a decomposition $\pi_1(M) \cong A*_{\langle c\rangle}B,$ where $c$ is homotopic to a simple closed curve on $S_1$ or $S_2$.

A homomorphism $\rho:\Gamma\ra \pslc$ is separable-stable if every
geodesic defined by a separable element in the Cayley graph of
$\Gamma$ is mapped to a uniform quasi-geodesic in $\H^3$. See Section
\ref{sep-stab} for a precise definition.  For any discrete and faithful representation $\rho: \pi_1(M) \rightarrow \pslc$ we let $N_\rho$ denote the corresponding hyperbolic $3$-manifold obtained from taking the quotient $\H^3/\rho(\pi_1(M)).$

The main theorem of this article is:
\begin{thm}\label{thm:main}
Let $M$ be a nontrivial hyperbolizable compression body without toroidal boundary components that is not the boundary connect sum of $S_1 \times I$ and $S_2 \times I,$ where $S_1$ and $S_2$ are closed surfaces with genus at least two.  Let $\rho$ be a discrete and faithful representation of $\pi_1(M)$ into $\pslc$.  Then, $\rho$ is separable-stable if and only if in $N_\rho,$ each component of the parabolic loci and each ending lamination is disc-busting.
\end{thm}

A measured lamination $\lambda$ is \emph{disc-busting} if there exists $\eta > 0$ such that $i(\lambda, m) \geq \eta$ for all unweighted meridians $m$. Otherwise it is \emph{disc-dodging}.  We say that an ending lamination is disc-busting if it is the support of a measured lamination that is disc-busting.  Such a lamination is used by Jeon-Kim-Ohshika-Lecuire in \cite{JKO} for the characterization of primitive stable representations of free groups, and has a root in doubly incompressible laminations (see Kim-Lecuire-Oshika \cite{KLO} or Lecuire \cite{Lec}).
In \cite{Lee2}, the second author showed that representations obtained by pinching Masur domain curves or Masur domain laminations on $\partial M$ are separable-stable.

For the case where $M$ is the boundary connect sum of $S_1 \times I $ and $S_2 \times I$ where $S_1$ and $S_2$ are closed surfaces, we establish the following analogous characterization of the set of discrete and faithful representations that are separable-stable.

\begin{thm}\label{thm:special case}
Let $M$ be the boundary connect sum of $S_1 \times I$ and $S_2 \times I,$ where $S_1$ and $S_2$ are closed surfaces with genus at least two.  Let $\rho: \pi_1(M) \rightarrow \pslc$ be a discrete and faithful representation.  Then, $\rho$ is separable-stable if and only if in $N_\rho,$ each component of the parabolic loci and each ending lamination is annulus-busting.
\end{thm}
A measured lamination $\lambda$ is {\it annulus busting} if there exists $\eta>0$ such that $i(\lambda, \partial A) \geq \eta$ for any essential annulus $A.$

\section{Preliminaries}

\subsection{Compression bodies}\label{sec:compbody}

A \emph{compression body} is a compact, orientable, irreducible $3$-manifold $M$ with a boundary component, $\partial_e M,$ called the \emph{exterior boundary}, such that the inclusion $i: \partial_e M \hookrightarrow M$ induces a surjection on the level of fundamental groups $i_*:\pi_1(\partial_e M) \twoheadrightarrow \pi_1(M)$. The other boundary components are called \emph{interior boundary components}.  Equivalently, a compression body is the boundary connect sum of a $3$-ball, a collection of trivial $I$-bundles over surfaces, and a handlebody, where the other components are connected to the $3$-ball along disjoint discs.  

The fundamental group of a compression body can be decomposed as a free product, $\pi_1(M) \cong G_1 * G_2 * \cdots * G_n,$ where $G_i$ is isomorphic to a closed surface group for $1 \leq i \leq k$ and $G_j$ is infinite cyclic $k<j \leq n$.  By Grushko's theorem (\cite{gru}) and Kurosh's subgroup theorem (\cite{kur}), any other decomposition of the fundamental group into a free product, $\pi_1(M) \cong H_1 * H_2 * \cdots H_m,$ where each factor is freely indecomposable, satisfies $n=m$ and $H_i \cong G_i$, up to re-ordering.  

A compression body is \emph{trivial} if $i_*:\pi_1(\partial_eM) \rightarrow \pi_1(M)$ is an isomorphism, i.e. if $M$ is a trivial $I$-bundle.  We say that $M$ is \emph{uniquely freely decomposable} if $M$ is the boundary connect sum of two trivial $I$-bundles over closed surfaces since, in this case, the decomposition of $\pi_1(M)$ is essentially unique.  

An \emph{essential disc} in $M$ is a properly embedded disc whose boundary is nontrivial in $\partial M.$  A compression body $M$ has the property that any splitting of $\pi_1(M)$ as a free product or as an HNN-extension over the trivial group can be realized by an essential disc in $M$ in the following sense.  Suppose that $D$ is a separating essential disc and $M-\mathcal{N}(D)=M_1 \sqcup M_2.$  If $M_i'=M_i \cup \mathcal{N}(D)$ for $i=1,2,$ then $\pi_1(M) \cong i_*(\pi_1(M_1')) * i_*(\pi_1(M_2')),$ where $i$ is inclusion and the base point is chosen to lie in $D.$  In this case, we say that $D$ realizes the splitting of the fundamental group.  As we can move the basepoint around, this is only well-defined up to conjugation.  If $D$ is a non-separating essential disc and $M_1 = M- \mathcal{N}(D),$ then $\pi_1(M) \cong i_*(\pi_1(M_1'))*_{\{1\}}.$  We say that $D$ realizes this HNN-extension.

\begin{lem}[Lee \cite{Lee2}, Lemma 2]\label{lem:freeprod}
Let $M$ be a compression body and let $\pi_1(M)= A * B$ be a nontrivial splitting of $\pi_1(M)$ into a free product.  Then, the splitting is realizable by an essential disc.
\end{lem}

\begin{lem}\label{lem:hnn}
Let $M$ be a compression body.
Any splitting of the form $\pi_1(M) \cong A*_{\{1\}}$ is realizable by an essential disc.
\end{lem}
\begin{proof}
Recall that $A*_{\{1\}}$ is isomorphic to $A*\mathbb{Z}.$  By Lemma \ref{lem:freeprod}, we can realize the latter splitting by an essential disc $D'$. Then, $D'$ separates $M,$ into two components $C_1$ and $C_2,$ where $\pi_1(C_1)$ is conjugate to $A$ and $C_2,$ is a compact, oriented, irreducible $3$-manifold with infinite cyclic fundamental group, i.e. a solid torus.  Then, the essential disc $D$ in the solid torus, realizes the original splitting.
\end{proof}

\subsection{Separable-stable representations}\label{sep-stab}
In a free group, an element is called \emph{primitive} if it can be completed to a free generating set and an element is \emph{separable} if it lies in a free factor.  The notion of separability is originally due to Stallings (\cite{sta}) in his study of separable sets
in free groups that generalized Whitehead's (\cite{Wh1}) study of primitive elements of free groups.  
Although the notion of a primitive element does not generalize to other freely decomposable groups, the notion of a separable element does.
\begin{dfn}If $M$ is a compression body that is not uniquely freely decomposable, then an element in $\pi_1(M)$ is \emph{separable} if it lies in a proper factor of a decomposition of $\pi_1(M)$ into a nontrivial free product.  If $M$ is the boundary connect sum of two trivial $I$-bundles over closed surfaces $S_1$ and $S_2$, then an element in $\pi_1(M)$ is \emph{separable} if it lies in a factor of a decomposition $\pi_1(M)=A*_{\langle c\rangle}B$ where $c$ is homotopic to a simple closed curve on $S_1$ or $S_2$
\end{dfn}

One can think of separability in the following geometric way.  If $M$ is not uniquely freely decomposable, by Lemma \ref{lem:freeprod}, an element $g$ of $\pi_1(M)$ lies in a proper factor of a free decomposition if and only if its associated curve in $M$ is homotopic to one that misses an essential disc.  Similarly, lying in a factor of a decomposition $\pi_1(M)=A*_{\langle c\rangle}B$ where $c$ is homotopic to a simple closed curve on $S_1$ or $S_2$ is equivalent to missing the corresponding essential annulus $c \times I,$ up to homotopy.  An \emph{essential annulus} in $M$ is a properly embedded incompressible annulus that is not homotopic, relative to its boundary, into the boundary of $M.$

Let $\Gamma$ denote the fundamental group of $M$ and let $S$ be a finite symmetric generating set for $\Gamma.$  The \emph{Cayley graph of $\Gamma$ with respect to $S$}, denoted $C_S(\Gamma)$, is a graph where the vertices are in one-to-one correspondence with the elements of $\Gamma$ and there is an edge between $g$ and $h$ if there exists an element $s$ in the generating set $S$ such that $gs=h$.  The group $\Gamma$ acts on $C_S(\Gamma)$ by left multiplication on the vertices.  Since $\Gamma$ is torsion-free and Gromov hyperbolic each element $g$ of $\Gamma$ has two fixed points on $\partial C_S(\Gamma)$.  For each $g$ in $\Gamma$, let $g_-$ and $g_+$ denote the repelling and attracting fixed points of $g$ acting on $\partial C_S(\Gamma)$.  Let $L_{S}(g)$ denote the set of geodesics connecting $g_-$ and $g_+$ and $\mathcal{S}_S$ denote the set of geodesics $l$ in $C_S(\Gamma)$ such that $l$ is contained in $L_{S}(g)$ for some separable element $g$.

Given a representation $\rho: \Gamma \rightarrow \pslc$ and a basepoint $x$ in $\H^3$, there exists a unique $\rho$-equivariant map $\tau_{\rho, x}: C_S(\Gamma) \rightarrow \H^3$ taking  the identity to $x$ and edges to geodesic segments.

\begin{dfn}
A representation $\rho: \Gamma \rightarrow \pslc$ is called \emph{$(K, A)$-separable-stable} if there exists a basepoint $x$ in $\H^3$ such that $\tau_{\rho, x}(l)$ is a $(K, A)$-quasi-geodesic for any $l$ in $\mathcal{S}_S.$
\end{dfn}
We will call $\rho$, a representation, \emph{separable-stable} if there exists $(K, A)$ such that $\rho$ is $(K, A)$-separable-stable.  Separable-stability is independent of the choice of basepoint in $\H^3$ and the choice of generators $S$ for $\Gamma$, and separable-stability is invariant under conjugation (see Lemma 16 in Lee \cite{Lee}).  The set of separable-stable representations forms a domain of discontinuity for the action of $\out(\pi_1(M))$ strictly larger than the interior of $AH(M)$.  Namely the following is true.   

\begin{thm}[Lee \cite{Lee2}]
Let $M$ be a nontrivial compression body without toroidal boundary components.  The outer automorphism group $\out(\pi_1(M))$ acts properly discontinuously on the set of separable-stable representations, which contains the interior of $AH(M)$ as well as points on the boundary of $AH(M).$  
\end{thm}

\subsection{The Whitehead graph}\label{white}
In this section, we describe Otal's (\cite{Ot}) generalization of the Whitehead graph for a free group to the Whitehead graph for a compression body, which we will use in proving the sufficient condition.

Let $M$ be a nontrivial compression body without toroidal boundary components.  A \emph{system of meridians} $\alpha$ on $M$ is a collection of pairwise nonisotopic and pairwise disjoint meridians $\alpha_i$ bounding discs $D_i$ such that $M \setminus (\cup \mathcal{N}(D_i))$ is a disjoint union of trivial $I$-bundles over closed surfaces, where $\mathcal{N}(D_i)$ is a regular neighborhood of $D_i$.

By Thurston's Hyperbolization Theorem (see Kapovich \cite{kap}) and Marden (\cite{mar}), we can fix $\sigma$ a convex cocompact representation of $\pi_1(M)$ such that $\overline N_\sigma= N_\sigma \cup \partial_C N_\sigma$ is homeomorphic to $M$, where $\partial_C N_\sigma$ is the conformal boundary of $N_\sigma.$  Let $\Lambda(\sigma)$ be the limit set of $\sigma(\Gamma).$  Let $\mu$ be a closed subset of $\Lambda(\sigma) \times \Lambda(\sigma)$ that is $\sigma(\Gamma)$-invariant and also invariant under switching the two factors.  In this paper, $\mu$ will be one of the following two examples.
\begin{itemize}
\item If $\gamma$ is a closed geodesic in $N_\sigma$, then $\mu_\gamma$ is the set of endpoints of all the lifts of $\gamma.$
\item If $\lambda$ is a lamination on $\partial M$ that is in tight position with respect to $\alpha$, then $\mu_\lambda$ is the set of endpoints of all lifts of all leaves of $\lambda.$
\end{itemize}
See the discussion before Proposition \ref{discbusting} for the definition of a lamination in tight position and Lemma \ref{endpoints} to see why $\mu_\lambda$ is well-defined.

Identify $\partial M$ with $\partial_C N_\sigma$.  Let $\alpha$ be a system of meridians bounding the discs $D=D_1 \cup \cdots \cup D_n$ and suppose $\overline N_\sigma-\mathcal{N}(D)$ is the disjoint union of $\Sigma_1 \times I, \ldots, \Sigma_k \times I$ where each $\Sigma_i$ is a closed surface of genus at least two. The \emph{Whitehead graph for $M$ of $\mu$ with respect to $\alpha$}, denoted $Wh(M, \alpha, \mu),$ is a collection of not necessarily connected graphs, $Wh(M, \alpha, \mu)^{\Sigma_1}, \ldots, Wh(M, \alpha, \mu)^{\Sigma_k},$  where the elements in the collection are in one-to-one
correspondence with the components of $\overline N_\sigma-\mathcal{N}(D)$.  In $\overline N_\sigma-\mathcal{N}(D)$, there are two copies $D_i^+$ and $D_i^-$ of each $D_i$ in $D$.  Given a component $\Sigma \times I$
of $\overline N_\sigma-\mathcal{N}(D)$, the vertices in the corresponding graph $Wh(M, \mu, \alpha)^{\Sigma}$
are in one-to-one correspondence with the components of  $D_i^{\pm}$  in the frontier
of $\Sigma \times I$.  Abusing notation, relabel the vertices $D_1, \ldots D_m.$  Fix a Jordan curve $C$ in $\Lambda(\sigma(\Gamma))$ that is invariant under a conjugate
of $\pi_1(\Sigma)$, which we will continue to denote $\pi_1(\Sigma)$.  Let $F$ denote the boundary component of
$\Sigma \times I$ coming from $\partial_eM$.  Fix a lift
$\widetilde{ \partial D_i}$ of each $\partial D_i$ in $\partial \H^3$ such that
$\widetilde {\partial D_i}$ lies in 
the component of the preimage of $F$ in $\partial \H^3$
containing $C$ on its boundary.  Let $U_i$ be the open set in $\partial \H^3$, bounded by $\widetilde{ \partial D_i}$ not containing $C$.  The edges from $D_{i}$ to $D_{j}$ will be in one-to-one correspondence with elements $g$ in $\pi_1(\Sigma)$ such that $(U_i, g U_j) \cap \mu$ is nonempty.  We will denote such an edge $(U_i, g U_j)$.  Although these edges are directed, for each edge from $U_i$ to $U_j$ labeled
$g$, there is an edge from $U_j$ to $U_i$ labeled $g^{-1}$.

\begin{dfn} [Otal] A connected component of $Wh(M, \alpha, \mu)^\Sigma$ is \emph{strongly connected} if there exists a cycle that represents a nontrivial element of
$\pi_1(\Sigma)$.  A connected component of $Wh(M, \alpha, \mu)^\Sigma$ has a \emph{strong cutpoint} if we can express the graph as the
union of two graphs $G_1$ and $G_2$ that intersect in a single vertex such that
either $G_1$ or $G_2$ is not strongly connected.
\end{dfn}

We take the convention that a cycle $$(U_{i_1}, g_1U_{i_2}), (U_{i_2},g_2U_{i_3}),
\ldots, (U_{i_k},g_k U_{i_1})$$ corresponds to the group element $g_1\cdots g_k$.

Although we made several choices when defining $Wh(M, \alpha, \mu)$, the two properties defined above, strong connectedness and the presence of a strong cutpoint are independent of these choices (see Section 3.1.2 in Lee \cite{Lee2}).

\subsection{Laminations in tight position}

Otal in \cite{Ot}, also, gave a dichotomy between the Whitehead graph of a separable element and the Whitehead graph of a disc-busting lamination in tight position with respect to a system of meridians.  In this section, we define and collect the relevant properties of a lamination in tight position.  The results in this section are essentially all due to Otal in \cite{Ot}.  Since \cite{Ot} is difficult to obtain for each result we either give an alternate source containing a proof or we provide a proof.

A lamination $\lambda$ is in \emph{tight position with respect to a system of meridians $\alpha$} if there are no waves disjoint from $\lambda$ where a \emph{wave} is an arc $k$ satisfying the following two conditions:
\begin{enumerate}
\item $k$ has endpoints on $\alpha$, but its interior is disjoint from $\alpha$ and,
\item $k$ is homotopic, relative to its endpoints, in $M$ but not in $\partial M$ into $\alpha.$
\end{enumerate}

For any disc-busting lamination $\lambda$ we can find a system of meridians $\alpha$ such that $\lambda$ is in tight position with respect to $\alpha.$  Otal proved this for Masur domain laminations but his proof applies in this more general case.

\begin{prop}[Otal \cite{Ot}, Theorem 1.3] \label{discbusting} If $\lambda$ is a disc-busting lamination, there exists a system of meridians
$\alpha$ with respect to which $\lambda$ is in tight position.
\end{prop}

\begin{proof}
Suppose that $\lambda$ is not in tight position with respect to a system of meridians $\alpha$.  Let $k$ be a wave and suppose its endpoints lie on $\alpha_i.$  Then the endpoints of $k$ split $\alpha_i$ into two arcs $k'$ and $k''$ such that $\alpha_i=k' \cup k''$ and $k' \cap k''$ are the endpoints of $k.$  Since $k$ is homotopic into $\alpha_i$ through $M$ but not in $\partial M,$ we have that $\alpha_i'=k \cup k'$ and $\alpha_i''=k \cup k''$ are both meridians.  Since $\lambda$ is disc-busting, $i(\lambda, \alpha_i')>\eta$ and $i(\lambda, \alpha_i'')>\eta.$  By construction both $\alpha_i'$ and $\alpha_i''$ are disjoint from $\alpha_j$ for any $j.$
Replacing $\alpha_i$ with either $\alpha_i'$ or $\alpha_i''$ will produce a new collection of meridians $\alpha'$ and $\alpha''$ respectively, whose intersection number with $\lambda$ is at least $\eta$ less than $i(\lambda, \alpha).$

We claim that at least one of these new collections of meridians $\alpha'$ or $\alpha''$ is a system of meridians.  By construction $\alpha_i, \alpha_i' $ and $\alpha_i''$ bound a pair of pants on $\partial M.$  Since they are all meridians and $M$ is irreducible, the discs they bound, $D_i, D_i'$ and $D_i'',$ bound a ball in $M$.
So $M-(D_i \cup D_i' \cup D_i'')$ contains a component that is a ball $B$.  We will use a $+$ superscript to denote the copy of $D_i$ (or $D_i', D_i''$) that lies on the boundary of the ball, i.e. $D_i^+, D_i'^+, D_i''^+$ lie on the boundary of $B.$  Since $M-D$ is a collection of trivial $I$-bundles over closed surfaces, $M-(D \cup D_i' \cup D_i'')$ must consist of a collection of trivial $I$-bundles over closed surfaces and two balls, one of which is the ball $B$ described above.  The other ball must contain either $D_i'^-$ or $D_i''^-$ on its boundary (possibly both).  If it contains $D_i'^-,$ then $\alpha''$ will be a system of meridians, and vice versa.

So, replacing $\alpha_i$ with one of $\alpha_i'$ and $\alpha_i''$ results in a system of meridians whose intersection number with $\lambda$ is at least $\eta$ less than $i(\lambda, \alpha).$  If $\lambda$ is not in tight position with respect to this new system of meridians, repeat this process.  Since $i(\lambda, \beta) \geq \eta$ for any system of meridians $\beta$, eventually, this process must terminate.
\end{proof}

An important property of disc-busting, minimal laminations in tight position with respect to a system of meridians $\alpha$ is that any lift of a leaf has well-defined endpoints.

\begin{lem}[Kleineidam-Souto \cite{KS}, Lemma 1 or Otal \cite{Ot}, Lemma 1.9] \label{endpoints}
Let $M$ be a nontrivial compression body.  Let $\sigma:\pi_1(M) \rightarrow \pslc$ be a convex cocompact representation uniformizing $M$.  Let $\lambda$ be a lamination on $\partial_C N_\sigma$ in tight position with respect to $\alpha.$  Let $l$ be a leaf $\lambda$ such that any half leaf of $l$ intersects $\alpha.$  Then, any lift $\tilde l$ of $l$ to $\widetilde{\partial_C N_\sigma}$ has two well-defined endpoints in $\Lambda(\sigma).$
\end{lem}

Propositions \ref{discbusting} and Lemma \ref{endpoints} imply that if $\lambda$ is a disc-busting ending lamination, then $\mu_\lambda$, the set of endpoints of lifts of leaves of $\lambda,$ is well-defined.  In particular, we will refer to the Whitehead graph of $\lambda$ as the Whitehead graph of $\mu_\lambda.$

Otal proved the following two results, which provide a dichotomy between separable elements and disc-busting laminations.
\begin{prop}[Otal \cite{Ot}, Proposition A.3 or see Lee \cite{Lee2}, Proposition 18] \label{separableelement} Let $M$ be a nontrivial compression body that is not the boundary connect sum of two trivial $I$-bundles over closed surfaces. If $g$ is separable, then for any system of meridians $\alpha$ some connected component of $Wh(M, \alpha, \mu_g)$ is either not strongly connected, or has a strong cut-point.
\end{prop}

\begin{prop}[Otal \cite{Ot}, Proposition A.5 or see Lee \cite{Lee2}, Proposition 19]\label{otal}
If $\lambda$ is a disc-busting lamination in tight position with respect to a system of meridians $\alpha,$ then the Whitehead graph of $\lambda$ with respect to $\alpha$ is strongly connected and without strong cutpoints.
\end{prop}

In \cite{Ot}, Otal states Proposition \ref{otal} only for Masur domain laminations, but in fact his proof applies in this more general case.

\section{Other homeomorphism types}
In general, for $\rho: \pi_1(M) \rightarrow \pslc,$ a discrete and faithful representation, the associated hyperbolic $3$-manifold $N_\rho$ will be homotopy equivalent but not necessarily homeomorphic to $M.$  Since Otal's results only apply in the case when $M$ is a compression body, in this section we associate a Whitehead graph to any hyperbolizable $3$-manifold $M'$ homotopy equivalent, but not necessarily homeomorphic to $M$.  We can do this by using the characteristic compression body introduced by Bonahon in \cite{bon}.  We also describe how to compare Whitehead graphs for $M$ and $M'$.  
Let $A(M)$ denote the set of compact, orientable, irreducible, atoroidal (unmarked) $3$-manifolds homotopy equivalent to $M$, up to homeomorphism.
Let $[M_0], \ldots, [M_n]$ be the elements in $A(M),$ where $M_0$ is homeomorphic to $M.$  For each $M_i$ fix a convex cocompact representation $\sigma_i$ such that $N_{\sigma_i}$ is homeomorphic to the interior of $M_i.$
Let $\tau_i: C_S(\Gamma) \rightarrow \H^3$ be an orbit map for $\sigma_i$ as in the definition of separable-stable representations where $\Gamma=\pi_1(M)$.  Then, $\tau_i$ is a quasi-isometric embedding and there is a continuous extension, $\partial \tau_i: \partial C_S(\Gamma) \rightarrow \Lambda(\sigma_i)$ which is a homeomorphism. We suppressed the basepoint in $\tau_i$ since $\partial \tau_i$ is independent of the choice of basepoint.  Then, let
$$
T_i=\partial \tau_0 \circ \partial \tau_i^{-1}: \Lambda(\sigma_i) \rightarrow \Lambda(\sigma_0)
$$
be the homeomorphism from $\Lambda(\sigma_i)$ to $\Lambda(\sigma_0).$

In order to compare the Whitehead graph of $M_0$ and $M_i$ we need a way to compare systems of meridians on the two manifolds.  Define a map $F_i: \mathcal{M}_i \rightarrow \mathcal{M}_0$ from the set $\mathcal{M}_i$ of unweighted meridians on $M_i$ to the set $\mathcal{M}_0$ of unweighted meridians on $M_0$ in the following way (see Figure \ref{pic}).

\begin{figure}
\begin{center}
\includegraphics[width=130mm]{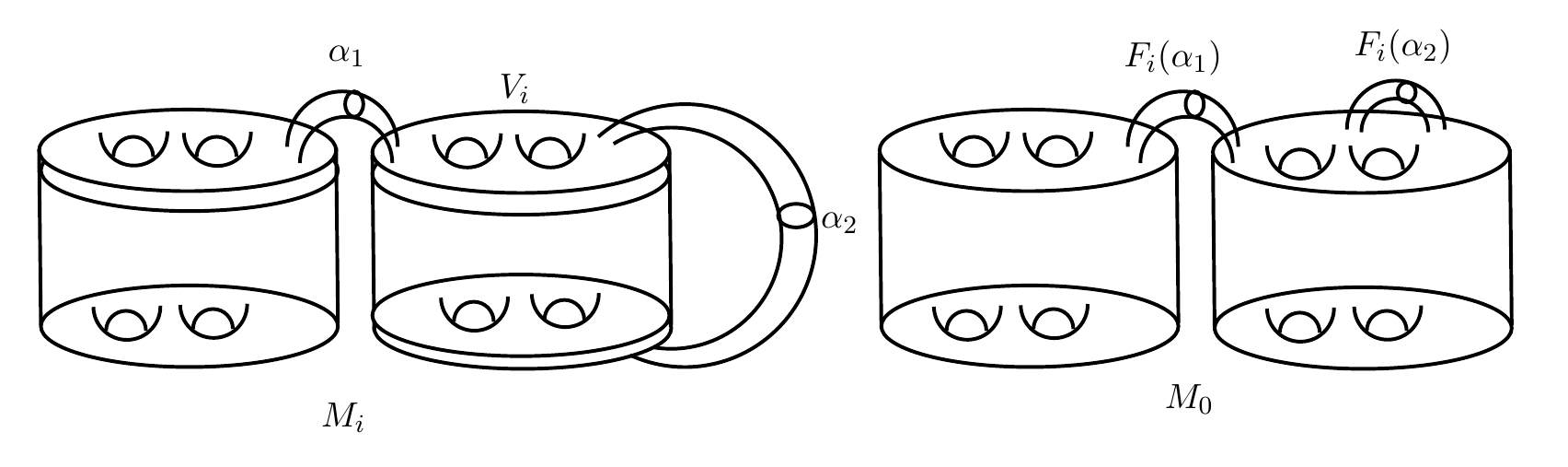}
\caption[Maps on meridians]{The system of meridians $\{\alpha_1, \alpha_2\}$ on the characteristic compression body $V_i$ of $M_i$ is mapped to the system of meridians $\{F_i(\alpha_1),F_i(\alpha_2)\}$ on $M_0.$}
\label{pic}
\end{center}
\end{figure}

Let $D_\alpha$ denote the essential disc bounded by $\alpha.$  First suppose that $\alpha$ is a meridian in $M_i$ such that $D_\alpha$ separates $M_i.$  Then, $D_\alpha$ realizes a splitting $\pi_1(M) = A*B,$
up to conjugation.
In $M_0$ we
can realize the splitting by an essential disc $D'$.  Let $F_i(\alpha)$ be the boundary of $D'.$  This is well-defined up to isotopy.  Indeed if we consider $F_i(\alpha)$ on the conformal boundary of $N_{\sigma_0},$
then some lift $\widetilde{ F_i(\alpha)}$ separates $\partial \mathbb{H}^3$
 into two components one of which contains all the attracting fixed points of any element whose reduced word representation starts with an element in $A$ and the other contains all the attracting fixed points of any element whose reduced word representation starts with an element in $B.$
  Since such points are dense in $\Lambda(\sigma_0),$ this determines $\widetilde{F_i(\alpha)},$ up to isotopy on $ \partial \widetilde M \subset \partial \H^3.$

Now suppose that $D_\alpha$ is non-separating.
Then, $D_\alpha$ realizes an HNN extension, $\pi_1(M) \cong A*_{\{1\}}.$  By Lemma \ref{lem:hnn}, we can realize such a splitting in $M_0$ by an essential disc $D'$.  As before, let $F_i(\alpha)$ be the boundary of $D'.$  Again, to see that $F_i(\alpha)$ is well-defined notice that some lift $\widetilde{F_i(\alpha)}$ separates $\partial \H^3$ into two components one of which contains all the attracting
fixed points of any element whose reduced word representation starts with an element of $A$ or $t$ and the other contains all the attracting fixed points of any element whose reduced word representation starts with $t^{-1},$ where $t$ is the new generator associated to the HNN-extension.

Any manifold $M$ with compressible boundary has a \emph{characteristic compression body} $V$, unique up to isotopy that has the following properties (see Bonahon \cite{bon} or McCullough-Miller \cite{mm}).  The exterior boundary of $V$ coincides with $\partial M$.  The closure of $M-V$ contains no essential compression discs and no component of $\overline{M-V}$ is a 3-ball.

Let $V_i$ denote the characteristic compression body of $M_i,$ with any trivial compression body components removed.

\begin{lem}\label{lem:mermap}

Let $\alpha_1, \ldots \alpha_n$ be a system of meridians on $V_i.$
Then, the collection of meridians $F_i(\alpha_1), \ldots, F_i(\alpha_n)$ is a system of meridians on $M_0.$
\end{lem}
\begin{proof}
First we will show that the meridians $F_i(\alpha_1), \ldots, F_i(\alpha_n)$ can be realized such that they are pairwise disjoint and pairwise nonisotopic.
To see that we can realize $F_i(\alpha_j)$ and $F_i(\alpha_k)$ disjointly, observe that since $\alpha_j$ and $\alpha_k$ are disjoint, any two lifts $\tilde \alpha_j$ and $\tilde \alpha_k$, have the property that there exist $U_j$ a component of $\partial \H^3 - \tilde \alpha_j$ and $U_k$ a component of $\partial \H^3 - \tilde \alpha_k$ such that $U_j \cap \Lambda(\sigma_i)$ and $U_k \cap \Lambda(\sigma_i)$ are each non-empty and disjoint from each other.  There exists lifts $\widetilde{F_i(\alpha_j)}$ and $\widetilde{F_i(\alpha_k)}$ and components $U_j'$ and $U_k'$ of $\partial \H^3 - \widetilde{F_i(\alpha_j)}$ and $\partial \H^3 - \widetilde{F_i(\alpha_k)},$ respectively, such that
$$T_i (U_j \cap \Lambda(\sigma_i))=U_j' \cap \Lambda(\sigma_0)$$ and
$$T_i(U_k \cap \Lambda(\sigma_i))=U_k' \cap \Lambda(\sigma_0).$$
Then, we can homotope, $F_i(\alpha_j)$ and $F_i(\alpha_k)$ to be disjoint.

To see that $F_i(\alpha_j)$ and $F_i(\alpha_k)$ are not isotopic, notice that if they were isotopic, then there exists lifts $\widetilde{F_i(\alpha_j)}$ and $\widetilde{F_i(\alpha_k)}$ that bound a region in $\partial \H^3$ containing no limit points.  Then, the same would be true about the corresponding lifts $\tilde \alpha_j$ and $\tilde \alpha_k,$ which contradicts that they are not isotopic.

Let $D_{F_i(\alpha_j)}$ be the disc bounded by $F_i(\alpha_j)$ in $M_0.$    Finally, we want to see that $$M_0 -( \bigcup_j \mathcal{N}(D_{F_i(\alpha_j)}))$$ is a collection of trivial $I$-bundles over closed surfaces.  We will first show that $M_i - \cup \mathcal{N}(D_{\alpha_j})$ is a collection of trivial $I$-bundles over closed surfaces.  Recall that $M_i=V_i \cup W_i,$ where $V_i$ is the characteristic compression body and $W_i$ has incompressible boundary.   By definition, $V_i - \cup \mathcal{N}(D_{\alpha_j})$ is a collection of trivial $I$-bundles over closed surfaces.

We claim that the components of $W_i$ are all trivial $I$-bundles over closed surfaces.  Let $C$ be a component of $W_i,$ and let $T$ be a component of $\partial C.$  Since $C$ has incompressible boundary and cannot be a $3$-ball, we have that $\pi_1(C)$ lies in a surface group factor of some maximal free decomposition of $\pi_1(M)$ i.e. $\pi_1(C)$ is a subgroup of some $\pi_1(S)$.
The image of $\pi_1(T)$ must be finite index in $\pi_1(S)$ since they are both closed surface groups.  This implies that $\pi_1(T)$ is finite index in $\pi_1(C),$ which in turn implies that $C$ is an $I$-bundle (Theorem 10.5 in Hempel \cite{hem}). So $C$ is either a trivial $I$-bundle or twisted $I$-bundle.  It remains to check that $C$ cannot be a twisted $I$-bundle.
 If $C$ were a twisted $I$-bundle, then the fundamental group of $C$ would be the fundamental group of a non-orientable surface, but this is impossible as $\pi_1(C)$ is finite index in $\pi_1(S).$
So, we have shown that the components of $W_i$ are all trivial $I$-bundles over surfaces.  Now $M_i-\cup \mathcal{N}(D_{\alpha_j})$ is formed by gluing $W_i$ and $V_i - \cup \mathcal{N}(D_{\alpha_j})$ along their boundaries, so the result will be a collection of trivial $I$-bundles over closed  surfaces.

Now, consider $C_0$ a component of $M_0- \cup D_{F_i(\alpha_j)}$.  We want to see that $C_0$ is a trivial $I$-bundle over a closed surface.  If not, then either $C_0$ is a $3$-ball, or $C_0$ has compressible boundary.  In the former case, let $F_i(\alpha_1), \ldots, F_i(\alpha_n)$ be the components of $F_i(\alpha)$ that lie on the boundary of $C_0$ (notice that we can have two copies of $\alpha_j$ on the boundary of $C_0$, relabel accordingly).  Fix a lift $\widetilde C_0$ of $C_0.$  Then, the corresponding lifts $\widetilde{F_i(\alpha_1)}, \ldots, \widetilde{F_i(\alpha_n)}$ bound a connected region in $\partial \H^3$ containing no limit points.
 The corresponding lifts of $\tilde \alpha_1, \ldots, \tilde \alpha_n$ of $\alpha$ would have the same property.  This implies that there is a $3$-ball component of $M_i - \cup D_{\alpha_j},$  a contradiction.

If $C_0$ has compressible boundary, then there exists a closed curve $\gamma$ in $C_0$ such that $\gamma$ does not lie in any surface group factor of a free decomposition of $\pi_1(M).$  Moreover, any lift $\tilde \gamma$ of $\gamma$ has the property that both endpoints lie in the same component of $\partial \H^3 - \widetilde{F_i(\alpha_j)}$ for any lift of any $\alpha_j,$ since $\gamma$ misses any $F_i(\alpha_j).$  Then, we would also have that $\tilde \gamma$ has both endpoints in the same component of $\partial \H^3 - \tilde \alpha_j.$  So, $\gamma$ misses all the meridians $\alpha_j,$ a contradiction, since $M_i - \cup D_{\alpha_j}$ consists of trivial $I$-bundles over closed surfaces.
\end{proof}

Using Lemma \ref{lem:mermap} we can compare the Whitehead graphs of $V_i$ and $M_0.$  Recall from Section \ref{white} the definition of the Whitehead graph of a compression body.  For $\mu \subset \Lambda(\sigma_i) \times \Lambda(\sigma_i),$ we will abuse notation, and let $T_i(\mu)$ denote the image of $\mu$ under $T_i \times T_i:  \Lambda(\sigma_i) \times \Lambda(\sigma_i) \rightarrow  \Lambda(\sigma_0) \times \Lambda(\sigma_0).$  

\begin{lem} \label{Whotherhomeo}
Let $M_i$ lie in $A(M)$ and let $V_i$ be the characteristic compression body of $M_i.$  If $Wh(V_i, \alpha, \mu)$ is strongly connected and without strong cutpoint, then $Wh(M_0, F_i(\alpha), T_i (\mu))$ is strongly connected and without strong cutpoint.
\end{lem}

\begin{proof}
Recall that we have a $\Gamma$-equivariant homeomorphism $T_i: \Lambda(\sigma_i) \rightarrow \Lambda(\sigma_0).$  Although $T_i$ is only defined on the limit set of $\sigma_i,$ it induces a natural map on lifts of $\alpha_1, \ldots, \alpha_n$ to lifts of $F_i(\alpha_1), \ldots F_i(\alpha_n).$  For example, given a lift $\tilde \alpha_j$ of $\alpha_j$, there exists a corresponding lift $\widetilde{F_i(\alpha_j)}$ of $F(\alpha_j)$ that partitions $\Lambda(\sigma_0)$ in the same way as $\tilde \alpha_j$ via $T_i.$  We'll abuse notation to denote such a lift $\widetilde{F(\alpha_j)}$ as $T_i(\tilde \alpha_j).$
More precisely, suppose that $\partial \H^3 - \tilde \alpha_j = V_1 \sqcup V_2.$ Then, $T_i(\tilde \alpha_j)$ is the lift $\widetilde{F_i(\alpha_j)}$ of $F_i(\alpha_j)$ such that if $\partial \H^3 - \tilde F_i(\alpha_j) = W_1 \sqcup W_2,$ then $T_i(V_k \cap \Lambda(\sigma_i)) = W_k \cap \Lambda(\sigma_0)$ where $k=1,2,$ up to switching $W_1$ and $W_2.$

Suppose that $\Sigma \times I$ is a component of $V_i - \cup D_{\alpha_j}.$  Recall that when defining the Whitehead graph, we fixed $C$ a Jordan curve in $\partial \H^3$ invariant under $\pi_1(\Sigma)$.  We let $F$ denote the boundary component of $\Sigma \times I$ intersecting $\alpha$ and we fixed the lift $\tilde F$ of $F$ whose boundary is $C.$  If $\tilde \alpha_j$ corresponds to a vertex of $Wh(V_i, \alpha, \mu)^\Sigma,$ then map it to $T_i(\tilde \alpha_j).$

We want to show that this is a well-defined map on the vertices.  First notice that since $\{F_i(\alpha_j)\}$ is a system of meridians, $M_0-\cup D_{F_i(\alpha_j)}$ consists of trivial $I$-bundles over surfaces.  One of these components $\Sigma' \times I$ must have fundamental group conjugate to $\pi_1(\Sigma),$ by the uniqueness of the free decomposition of $\pi_1(M)$ (see Section \ref{sec:compbody}).  Let $C'$ be the Jordan curve invariant under $\pi_1(\Sigma)$.  Note that $T_i(C)=C'.$  Let $F'$ be the component of $\Sigma \times I$ intersecting $F_i(\alpha)$ and let $\tilde F'$ denote the component of the preimage of $F'$ with $C'$ on its boundary.  Then, we claim that $T_i(\tilde \alpha_j)$ lies on $\tilde F'.$  Indeed, this is clear from the fact that $T_i(\tilde \alpha_j)$ partitions $\Lambda(\sigma_0)$ in the same way as $\tilde \alpha_j$ via $T_i,$ as described above.   So we see that the map on the vertices is well-defined.

Now suppose that $(U_j, gU_k)$ is an edge in $Wh(V_i, \alpha, \mu).$  This implies that the intersection of $\mu$ and $(U_j \times gU_k)$ is nonempty.  By construction, this is true if and only if the same is true for $T_i(\mu)$ and $T_i(U_j) \times gT_i(U_k).$  So, we have a bijection between the Whitehead graph of $Wh(V_i, \alpha, \mu)$ and that of $Wh(M_0, F_i(\alpha), T_i(\mu))$ where an edge in $Wh(V_i, \alpha, \mu)$ labeled $g$ is mapped to an edge in $Wh(M_0, F_i(\alpha), T_i(\mu))$ with the same label.  In particular, we see that $Wh(V_i, \alpha, \mu)$ is strongly connected and without cutpoint if and only if the same is true for $Wh(M_0, F_i(\alpha), T_i(\mu)).$
\end{proof}

\section{Cannon-Thurston maps}
Recently, Mj (\cite{Mj8}) has proved the existence of Cannon-Thurston maps for general Kleinian groups.  In particular, given a Kleinian group $\Gamma$ and a discrete and faithful representation $\rho: \Gamma \ra \pslc$, he showed that there exists a continuous extension $\bar \tau_{\rho, x}: C_S(\Gamma) \cup \partial C_S(\Gamma) \rightarrow \H^3 \cup \partial \H^3$ of $\tau_{\rho, x}.$  Moreover, he gives a characterization, which we describe below, of which points are mapped non-injectively by $\bar \tau_{\rho, x}$.

Before stating Mj's result we recall some facts from hyperbolic geometry that we will need.  We will restrict to the case that $\Gamma$ has no non-cyclic abelian subgroups since this is the case in this paper.  

There exists a constant $\mu_3,$ called the Margulis constant, such that for any hyperbolic $3$-manfiold, and any $\epsilon > \mu_3,$ each component of the $\epsilon$-thin part of $N$ is a metric neighborhood of a closed geodesic or a parabolic cusp homeomorphic to $S^1 \times \mathbb{R} \times (0,\infty).$  Given $\epsilon < \mu_3$ the \emph{non-cuspidal part of $N_\rho$}, denoted  $(N_\rho)_0^\epsilon,$ is the submanifold of $N_\rho$ obtained by removing the noncompact components of the $\epsilon$-thin part of $N_\rho.$  

When $\rho(\Gamma)$ has parabolics,
$N_\rho=\H^3/{\rho(\Gamma)}$ has a \emph{relative compact core,} $C,$ that is a compact submanifold whose inclusion is a homotopy equivalence and that intersects each cusp neighborhood
along an annulus whose core curve we call a {\em parabolic curve} (McCullough \cite{mcc}).  Let $P$ denote the union of these annuli.  Moreover, by the Tameness (Agol \cite{ago} or Calegari-Gabai \cite{cal-gab}) and the uniqueness of cores (McCullough-Miller-Swarup \cite{mms}) $N_\rho$ is homeomorphic to the interior of $C$.

Each frontier component $S_i$ of $\partial C-P$ in $(N_\rho)_0^\epsilon$ faces an end
$E_i$ with a well-defined ending lamination $\lambda_i,$ obtained from taking the support of a projective limit of any sequence of simple closed geodesics exiting the end $E_i.$  The ending lamination $\lambda_i$ is a
filling minimal lamination on $S_i$ (Canary \cite{can}). 

Let $\lambda' \subset \partial C$ be the union of $\lambda_i$ and the
parabolic curves, and assume that each is disc-busting.  If we fix a convex cocompact structure $N_\sigma$ on $C$, then we can identify $\lambda'$ with a geodesic lamination on $\partial_C N_\sigma$. By Lemma \ref{endpoints} any lift $\tilde l$ of a leaf $l$ of $\lambda_i$ for any $i$ has two well defined endpoints at infinity on $\Lambda(\sigma(\Gamma))$.  Since $\sigma$ is convex cocompact, we can identify $\partial C_S(\Gamma)$ with $\Lambda(\sigma(\Gamma))$.  We
introduce a relation $\mathcal R$ for points in $\Lambda(\sigma(\Gamma))$ (and so also on $\partial C_S(\Gamma)$) such that $a \mathcal R b$ if and only if either $a$ and $b$
are the endpoints of a leaf of $\widetilde \lambda'$ or ideal vertices
of a complementary region of $\widetilde \lambda'$. In contrast to
the case of groups without parabolics, there are complementary
regions of $\widetilde \Lambda$ which are ideal polygons with
infinitely many sides. These are exactly regions touching lifts of
parabolic curves, which are isolated leaves. The relation $\mathcal
R$ may not be transitive. We define $\widetilde{\mathcal R}$ to be
the transitive closure of $\mathcal R$, which is an equivalence
relation.

Mj's theorem about the identified points of a Cannon-Thurston map
can be adapted to our case as follows.
\begin{thm}[Mj \cite{Mj8}]\label{Mj2}
Let $\rho:\Gamma\ra\pslc$ be a discrete and faithful representation.  
The identification of $C_S(\Gamma)$ with its image
under $\tau_{\rho,x}$ extends continuously to a map $\bar \tau_{\rho,x} :
C_S(\Gamma) \cup \partial C_S(\Gamma) \rightarrow \H^3 \cup
\partial \H^3$. Moreover, for $a,b\in\partial C_S(\Gamma)$, we have $\bar
\tau_{\rho, x}(a)=\bar \tau_{\rho, x}(b)$ if and only if $a\widetilde{\mathcal R} b$.
\end{thm}

\section{Sufficient condition}
In this section, we use the dichotomy between the Whitehead graph of a separable element and a disc-busting lamination to show the sufficient condition of Theorem \ref{thm:main}, namely that when $\pi_1(M)$ is not uniquely freely decomposable, a representation is separable-stable if each end invariant is disc-busting.  For the sufficient condition, we'll use Mj's result described in the previous section.  The following lemma, whose proof can be found in Lee \cite{Lee2} Lemma 21, sets the stage for using Mj's result.

\begin{lem}\label{notss} Let $\rho: \pi_1(M) \rightarrow \pslc$ be discrete and faithful.  If $\rho$ is not separable-stable, then there exists a sequence of separable elements $g_i$ such that if $g_i^+$ and $g_i^-$ are the endpoints of $g_i,$ then $\bar \tau_\rho(g_i^+)$ and $\bar \tau_\rho(g_i^-)$ converge, up to subsequence, to the same point in $\partial \H^3,$ but $g_i^+$ and $g_i^-$ converge to $z^+$ and $z^-$ in $\partial C_S(G)$ where $z^+ \neq z^-.$
\end{lem}

\begin{prop}\label{sufficient}
Let $M$ be a nontrivial compression body, without toroidal boundary components, that is not the boundary connect sum of two trivial $I$-bundles over closed surfaces.  Let $\rho: \pi_1(M) \rightarrow \pslc$ be a discrete and faithful representation.  If each parabolic locus and ending lamination of $N_\rho$ is disc-busting, then $\rho$ is separable-stable.
\end{prop}

Notice that the statement of the above proposition does not make any assumption about the homeomorphism type of $N_\rho.$

\begin{proof}
Suppose that $N_\rho$ is homeomorphic to the interior of $M_k$ where $[M_k] \in A(M).$  Let $\lambda_1, \ldots, \lambda_n$ be the union of the parabolic loci and ending laminations on $\partial M_k.$
If $\rho$ is not separable-stable, then by Lemma \ref{notss} there exists a sequence of separable elements $g_i$ such that the fixed points of $g_i$ in $\partial C_S(\Gamma)$ converge to $z^+$ and $z^-,$ but the endpoints of $\rho(g_i)$ converge to the same point in $\partial \H^3.$  By Theorem \ref{Mj2} we have that $\bar \tau_k(z^+)$ and $\bar \tau_k(z^-)$ are each an endpoint of a leaf of one of the end invariants, although not necessarily the same end-invariant.  Suppose that $\bar \tau_k(z^+)$ contains an endpoint of a leaf of $\lambda=\lambda_j.$

Let $\mu_\infty \subset \Lambda(\sigma_k) \times \Lambda(\sigma_k)$ be the set of limit points of $\{\mu_{\sigma_k(g_i)}\}.$  Then $\mu_\infty$ is $\sigma_k(\Gamma)$-invariant and also invariant under switching the two factors.  Moreover, $(\bar \tau_k(z^+), \bar \tau_k(z^-))$ lies in $\mu_\infty.$  This implies that $Wh(V_k, \alpha, \mu_\infty)$ contains $Wh(V_k, \alpha, \mu_\lambda)$ for any system of meridians $\alpha$ (see the proof of Proposition 22 in Lee \cite{Lee2}).

Fix a system of meridians $\alpha$ such that $\lambda$ is in tight position with respect to $\alpha.$  Then, $Wh(V_k, \alpha, \mu_\lambda)$ is strongly connected and without any strong cutpoints by Proposition \ref{otal}.

By Lemma \ref{Whotherhomeo}, the Whitehead graph $Wh(M_0, F_k(\alpha), T_k(\mu_\lambda))$ is strongly connected and without strong cutpoint. Observe that $Wh(V_k, \alpha, \mu_\lambda)$ is a finite graph since edges correspond to homotopy classes of arcs in $\lambda,$ so $Wh(M_0, F_k(\alpha), T_k(\mu_\lambda))$ is also finite.
Since the graph $Wh(V_k, \alpha, \mu_\infty)$ contains $Wh(V_k, \alpha, \mu_\lambda)$, for $i$ large enough $Wh(V_k, \alpha, \mu_{\sigma_k(g_i)})$ contains the Whitehead graph $Wh(V_k, \alpha, \mu_\lambda).$ This implies $Wh(M_0, F_k(\alpha), T_k(\mu_{\sigma_k(g_i)}))$ contains $Wh(M_0, F_k(\alpha), T_k(\mu_\lambda)),$ for $i$ large enough.  But then the Whitehead graph $Wh(M_0, F_k(\alpha), T_k(\mu_{\sigma_k(g_i)}))$ is strongly connected and without strong cutpoint.  Notice, that $T_k(\mu_{\sigma_k(g_i)})$ is the same as $\mu_{(\sigma_0(g_i))}.$  This contradicts Proposition \ref{separableelement}.
\end{proof}

\section{Necessary condition}
In this section, we show that when $\pi_1(M)$ is not uniquely freely decomposable, if $\rho: \pi_1(M) \rightarrow \pslc$ is discrete and faithful, then the condition that each end invariant of $N_\rho$ is disc-busting is also necessary.  Notice that if each end invariant is disc-busting, then, in particular, $N_\rho$ contains only one compressible boundary component.  We will make use of the following characterization of which discrete and faithful representations are separable-stable.

\begin{lem}[Lee \cite{Lee2}, Lemma 20]\label{compact}
Let $\rho: \pi_1(M) \rightarrow \pslc$ be discrete and faithful.  Then, $\rho$ is separable-stable if and only if every separable curve is homotopic to a closed geodesic and the set of separable geodesics is contained in a compact set of $N_\rho.$
\end{lem}

If $\lambda$ is a minimal lamination that is not a simple closed curve, then $S(\lambda),$ the \emph{supporting surface of} $\lambda,$ is the unique minimal compact subsurface with geodesic boundary containing $\lambda.$  Let $C(\lambda)$ be the unique, up to isotopy, maximum, simple, multi-curve of $S(\lambda)$ disjoint from $\lambda.$  The multi-curve $C(\lambda)$ contains $\partial S(\lambda)$.

\begin{prop}\label{necessary}
Let $M$ be a nontrivial compression body, without toroidal boundary components, that is not the boundary connect sum of two trivial $I$-bundles over closed surfaces.  Suppose $\rho:\pi_1(M)\ra\pslc$ is a discrete and faithful representation.  If $\rho$ is separable-stable, then each end invariant of $N_\rho$ is disc-busting.
\end{prop}

The proof of the necessary condition follows almost exactly the proof in the case when $M$ is a handlebody from Jeon-Kim-Ohshika-Lecuire \cite{JKO} (see section 5).  We include it here for completeness.

\begin{proof}
By Lemma \ref{compact}, if a representation maps a separable element of $\Gamma$ to a parabolic element in $\pslc,$ then the representation is not separable-stable.  Since any disc-dodging, closed curve is separable, if there is a disc-dodging parabolic curve in $N_\rho$, then $\rho$ is not separable-stable.  So we may assume that each parabolic curve is disc-busting.

Similarly, if there exists an ending lamination $\lambda$ that misses an essential disc $D$, then $\lambda$ is the limit of separable elements.  In particular, there exists a sequence of separable elements exiting every compact set of $N_\rho.$  By Lemma \ref{compact}, $\rho$ is not separable-stable.  So, we may also assume that $i(\lambda, \partial D)>0$ for any essential disc $D$ and any ending lamination $\lambda$ of $N_\rho.$

Then, Proposition \ref{necessary} follows from the following two lemmas whose proofs we include below.  We first explain how the necessary condition follows from the following two lemmas, and we include the proofs of the lemmas, afterwards.

\begin{lem}[Jeon-Kim-Ohshika-Lecuire \cite{JKO} Lemma 5.2]     \label{esan}
Let $M$ have compressible boundary.  For any essential annulus $A$ in $M$, there is a meridian which is disjoint from $A$.
\end{lem}

\begin{lem}[Jeon-Kim-Ohshika-Lecuire \cite{JKO}, Lemma 5.3] \label{limitannuli}
Let $\rho: \pi_1(M) \rightarrow \pslc$ be discrete and faithful.  Suppose that $\lambda$ is a disc-dodging ending lamination of $N_\rho$.  If some component $c$ of $C(\lambda)$ is disc-busting, then $\lambda$ is the unique minimal component in the Hausdorff limit of a sequence $\partial A_i,$ where $A_i$ is an essential annulus.
\end{lem}

Indeed, by Lemma \ref{limitannuli}, if $N_\rho$ has an ending lamination that is disc-dodging, then there exists a sequence of simple closed curves $\gamma_i$, that are homotopic to core curves of essential annuli, such that $\{\gamma_i\}$ exits every compact set.  By Lemma \ref{esan} these curves are separable.  Hence, there exists a sequence of separable curves exiting every compact set.  By Lemma \ref{compact}, the representation is not separable-stable.

\end{proof}

\begin{proof}[Proof of Lemma \ref{esan}]
Let $A$ be an essential annulus in $M.$  Since $M$ has compressible boundary there exists an essential disc $D.$  We can isotope $D$ so that there are no inessential intersections between $D$ and
$A$. If $D \cap A= \emptyset$, then the boundary of $D$ is a meridian disjoint from $A.$ If not, consider
an outermost arc $k$ in $D \cap A.$  Then, $k$ together with an arc in $\partial D$ bounds a disc $\Delta$.    If both of the endpoints of $k$ lie on the same component of
$\partial A$, then together with an arc in $\partial A,$ the arc $k$ cuts off a disc $\Delta'$ from $A,$ and $\Delta
\cup \Delta'$ is an essential disc which can be isotoped off $A$.
If $k$ connects two components of $\partial A$, then we can
boundary-compress $A$ along $\Delta$, and to get an essential disc disjoint from $A.$
\end{proof}

\begin{proof}[Proof of Lemma \ref{limitannuli}]
The proof of Lemma \ref{limitannuli} makes use of a result of Lecuire in \cite{Le} that states the following.  If $M$ is a compact hyperbolizable $3$-manifold, and there exists a lamination $\lambda' \subset \partial M$ such that
\begin{itemize}
\item[(A)] each closed leaf has weight at most $\pi$ and
\item[(B)] for any compressing disc $D$, we have $i(\lambda', \partial D) > 2\pi$
\end{itemize}
then, the following two conditions are equivalent.
\begin{enumerate}
\item[(C)] There exists $\eta$ such that $i(\lambda', \partial A) \geq \eta$ for any essential annulus $A.$
\item[(C')] If $r_1$ and $r_2$ are disjoint geodesic rays on $\partial M - \lambda'$ such that two lifts $\widetilde r_1$ and $\widetilde r_2$ share the same endpoint in $\partial \H^3,$ then they are asymptotic.
\end{enumerate}

The first step is finding a lamination containing $\lambda$ that satisfies the first two conditions (A) and (B).  Let $\lambda_0$ be the union of $\lambda$ and $C(\lambda),$ where we place a weight of $\pi$ on each closed curve.  If there exists an essential annulus $A$ that is disjoint from $\lambda_0,$ then cut $M$ along $A$ and let $M_1$ be the component containing $\lambda.$  We claim that $\lambda$ is still disc-dodging on $M_1.$  Indeed, suppose that $m_i$ is a sequence of meridians on $M$ such that $i(\lambda, m_i) \rightarrow 0.$  Then, as in the proof of Lemma \ref{esan} we can find a sequence of meridians $m_{1,i},$ disjoint from $A,$ such that $i(\lambda, m_{1,i}) \leq 2 i(\lambda, m_i).$  Hence, $\lambda$ is still disc dodging on $M_1.$  Let $\lambda_1$ be the union of $\lambda_0$ and the core curve or curves of $A,$ depending on whether $A$ separates $M$.  Put a weight of $\pi$ on each core curve that is added.  Notice that this new sequence of meridians also has the property that $i(\lambda_1, m_{1,i}) = i(\lambda_1 \cap S(\lambda), m_{1,i}).$

If there exists an essential annulus on $M_1$ disjoint from $\lambda_1,$ we repeat the procedure above.  Since there exists at most finitely many disjoint non-isotopic essential annuli on $M,$ this procedure must terminate.  Let $M_\infty$ be the manifold obtained at the end of this procedure, and let $\lambda_\infty$ be the corresponding lamination.  Then, any closed leaf has weight at most $\pi,$ by construction.  To see that $i(\lambda_\infty, \partial D) > 2\pi$ for any essential disc $D,$ it suffices to show that $\partial D$ intersects an element in $C(\lambda)$ twice.  Indeed, this implies that $i(C(\lambda), \partial D) \geq 2\pi$ and since $i(\lambda, \partial D) >0,$ we have that $i(\lambda_\infty, \partial D) > 2\pi.$
Since there is a disc-busting component $c$ of $C(\lambda),$ we know that $c$ intersects $D$ at least once. If $c$ intersects $D$ exactly once, then a regular neighborhood of $c \cup D$ is a solid torus, $V.$  We can form an essential disc missing $c$ by taking the closure of $\partial V - \partial M_\infty$, a contradiction since $c$ is disc-busting. So $\lambda_\infty$ satisfies conditions $(A)$ and $(B)$ above.  

Secondly, we want to find two geodesic rays satisfying the hypothesis of condition $(C')$ above. Namely, we want to find two disjoint, geodesic rays on $\partial M_\infty - \lambda_\infty$ that have lifts with a common endpoint in $\partial \H^3.$  Recall that there exists a sequence of meridians $m_i'$ such that $i(\lambda, m_i') \rightarrow 0$ and $i(\lambda_\infty, m_i') = i(\lambda_\infty \cap S(\lambda), m_i').$  Up to subsequence, $m_i'$ converges to a geodesic lamination $\mu,$ in the Hausdorff topology.  By Casson's criterion (see Casson-Long \cite{CL} or Theoreme B1 in Lecuire \cite{Le}), $\mu$ contains a homoclinic leaf $h.$  We want to show that $h$ has two half-leaves $h^+$ and $h^-$ that are disjoint from $\lambda_\infty.$  Then, if $c$ is the disc-busting component in $C(\lambda),$ we have that $h^+$ and $h^-$ will lie on $\partial M_\infty - c.$ Since $\partial M_\infty-c$ is an incompressible subsurface, any lift $\tilde h^{\pm}$ has a well-defined endpoint in $\partial_\infty \Gamma',$ where $\Gamma'$ is the subgroup of $\pi_1(M_\infty)$ corresponding to the subsurface $\partial M_\infty-c.$  Since $h$ is homoclinic, $h^+$ and $h^-$ must have the same endpoint.  

To see why two such half-leaves exist, first note that if $h$ intersects $\lambda$ transversely, then it contains an arc $k$ such that $\int_k d\lambda > 0.$  This implies that for $i$ large enough, $i(\lambda, m_i') \geq \int_k d\lambda,$ which is a contradiction since $i(\lambda, m_i') \rightarrow 0.$  So $h$ does not intersect $\lambda$ transversely.  It follows that $h$ has two half-leaves $h^\pm$ that are disjoint from $\lambda \cup C(\lambda).$  Similarly, if $h$ intersects a leaf $\gamma$ of $\lambda_\infty$ outside of $S(\lambda),$ then $i(\gamma, m_i')\geq \pi$ for $i$ large enough, a contradiction since $i(\lambda_\infty, m_i') = i(\lambda_\infty \cap S(\lambda), m_i').$  So $h$ has two half-leaves $h^\pm$ that are disjoint from $\lambda_\infty.$

Next, we want to see that $\lambda_\infty$ does not satisfy condition (C') above.  Indeed we want to show that if we take lifts $\widetilde h^+$ and $\widetilde h^-$ with the same endpoint, then they are not asymptotic. Suppose they are asymptotic and take a sequence of geodesic arcs $\widetilde k_n$ joining $\widetilde h^+$ and $\widetilde h^-$ such that the length of $\widetilde k_n$ goes to zero.  Let $k_n$ be the projection of $\widetilde k_n$ to $\partial M_\infty$.  Then $\int_{k_n} d\lambda \rightarrow 0$ since $\lambda$ is minimal and not a simple closed curve.  Since $\widetilde h^+$ and $\widetilde h^-$ are asymptotic and do not transversely intersect $\lambda$ we can assume that $k_n$ are homotopic relative to their endpoints and that the homotopy is through arcs transverse to $\lambda.$  Then, we must have $\int_{k_n} d\lambda =0$ for all $n$.  In particular, the arc $k_1$ is disjoint from $\lambda.$  Then, adjusting $h^{\pm}$ so that it begins at $k_1 \cap h$, if we take $\widetilde k_1 \cup  (\widetilde h \backslash \widetilde h^{\pm})$ this projects to a nontrivial curve $m$ on $\partial M_\infty$ disjoint from $\lambda$ that is trivial in $M_\infty$.  It follows from the proof of Dehn's Lemma that there is a compressing disc disjoint from $\lambda,$ a contradiction.

Since condition (C') is not true, we have that condition (C) is also not true.  In particular, there exists a sequence $A_i$ of essential annuli such that $i(\lambda_\infty, \partial A_i) \rightarrow 0.$  Since each component of $\lambda_\infty \backslash \lambda$ is a closed curve with weight $\pi$, we have that $i(\lambda_\infty \backslash \lambda, \partial A_i)=0$ for $i$ large enough.  Up to subsequence, $\partial A_i$ converges, in the Hausdorff topology, to a geodesic lamination $\nu.$  Then, $\nu$ and $\lambda$ cannot intersect transversely.  Since $i(\lambda, \partial A_i) \neq 0,$ we cannot have that $\nu$ and $\lambda$ are disjoint.  Since  $i(\lambda_\infty \backslash \lambda, \partial A_i)=0$ for $i$ large enough, we have that $i(\nu, C(\lambda))=0.$  In particular, $\nu$ lies on $S(\lambda) \backslash C(\lambda).$  Since $\lambda$ is filling on $S(\lambda) \backslash C(\lambda)$, we have that $\lambda$ and $\nu$ coincide as geodesic laminations. Since $\partial A_i$ is disjoint from all the essential annuli used to construct $M_\infty$,  we can view $\nu$ as a lamination on $\partial M$ and we are done.
\end{proof}

\section{The uniquely freely decomposable case}

This section deals with the missing case when $\pi_1(M)$ is uniquely freely decomposable, i.e. when $M$ is the boundary connect sum of $S_1 \times I$ and $S_2 \times I,$ where $S_i$ is a closed surfaces of genus at least two for $i=1,2.$  Here an element of $\pi_1(M)$ is separable if and only if it misses an essential annulus in one of the two surface group factors.  

By Lemma \ref{esan} for any essential annulus $A,$ there exists an essential disc disjoint from $A.$  Since there is only one essential disc, $D,$ up to isotopy, in fact, $A$ lies in one of the two surface group factors.  If $A$ is contained in $S_i \times I,$ then $A$ is either an essential annulus in $S_i \times I$ or homotopic (relative to its boundary) into the boundary of $S_i \times I.$  If the latter is true, since $A$ is essential in $M,$ if $M-D=M_1 \sqcup M_2$ and $A$ lies in $M_1,$ then $\partial A$ bounds an annulus in $\partial \overline{M_1}$ containing $D.$  In either case, if a curve or lamination misses $A,$ it misses an essential annulus in $S_1 \times I$ of $S_2 \times I.$ 

We will consider annulus-busting laminations on $\partial M$.  Recall that a measured lamination $\lambda$ is annulus-busting if there exists an $\eta$ such that $i(\lambda, \partial A) \geq \eta >0$ for any essential annulus $A.$

\begin{prop}
Let $M$ be the boundary connect sum of $S_1 \times I$ and $S_2 \times I,$ where $S_1$ and $S_2$ are closed surfaces of genus at least two.  Let $\rho: \pi_1(M) \rightarrow \pslc$ be discrete and faithful.  Then, $\rho$ is separable-stable if and only if each end invariant of $N_\rho$ is annulus-busting.
\end{prop}

\begin{proof}
We will start with the sufficient condition, and we will follow the general outline of Proposition \ref{sufficient}.  In this case, we have the following simplifications.  Firstly, $A(M)$ has only one element, i.e. any manifold homotopy equivalent to $M$ is homeomorphic to $M$.  Secondly, we have only one meridian $\alpha$.  If $\lambda$ is annulus-busting, then it must intersect $\alpha,$ and so it is automatically in tight position with respect to $\alpha.$

Now suppose that $\rho: \pi_1(M) \rightarrow \pslc$ is a discrete and faithful representation such that each end invariant and parabolic locus of $N_\rho$ is annulus-busting.  If $\rho$ is not separable-stable, by Lemma \ref{notss}, there exists a sequence of separable elements $g_i$ with endpoints $g_i^{\pm}$ such that $g_i^+ \rightarrow z^+, g_i^- \rightarrow z^-$ and $\bar \tau(z^+)=\bar \tau(z^-).$  Then, by Theorem \ref{Mj2}, the image $\bar \tau(z^+)$ is an endpoint of one of the end invariants $\lambda$.  Here $\lambda$ can either be an ending lamination or a simple closed curve that is mapped to a parabolic element under $\rho.$  As before, this implies that for $i$ large enough $Wh(M, \alpha, \mu_{g_i})$ contains $Wh(M, \alpha, \mu_\lambda).$  If $\lambda$ is annulus-busting, then $Wh(M, \alpha, \mu_\lambda)$ intersects each essential annulus in the following sense.

Suppose that $A$ is an annulus in $S_i \times I.$  Let $\partial A = c_1 \sqcup c_2$.  In defining the Whitehead graph, we fixed lifts $\widetilde S_i$ of $S_i$.  If we take a lift $\widetilde c_1$ of $c_1$ and the lift $\widetilde c_2$ of $c_2$ with the same endpoints as $\widetilde c_1$, then $\widetilde c_1 \cup \widetilde c_2$ forms a loop in $\partial \H^3$.
 We will say that an edge $e = (U_i, gU_j)$ intersects $A$ if there exists lifts $\widetilde c_1$ and $\widetilde c_2$ in $\widetilde S_i$ as above such that $U_i$ and $gU_j$ lie in different components of $\partial \H^3 - (\widetilde c_1 \cup \widetilde c_2)$.

Since $Wh(M, \alpha, \mu_{g_i})$ contains $Wh(M, \alpha, \mu_\lambda)$ for $i$ large enough, we have that $Wh(M, \alpha, \mu_{g_i})$ ``intersects'' any essential annulus.  This implies that the geodesic representative $\gamma_i$ of $g_i$ intersects any essential annulus, a contradiction.

For the necessary condition, suppose that $\rho:\pi_1(M) \rightarrow \pslc$ is a discrete, faithful and separable-stable representation and some end invariant of $N_\rho$ is annulus-dodging.  If a parabolic curve $c$ is annulus-dodging, then $c$ is separable, a contradiction to Lemma \ref{compact}.  So suppose that all parabolic curves are annulus-busting, but there is an ending lamination $\lambda$ that is annulus-dodging. If $\lambda$ is disjoint from an essential annulus $A,$ then we can find a sequence of separable elements approaching $\lambda,$ a contradiction to Lemma \ref{compact}.  So we can assume that $(\lambda, A) >0$ for any essential annulus.  Since $\lambda$ is annulus-dodging, there exists a sequence $A_i$ of essential annuli such that $i(\lambda, \partial A_i) \rightarrow 0.$  Since each essential annulus misses the essential disc, $D,$ up to a subsequence, we can assume that $A_i$ is contained in one of the two surface group factors for all $i.$  Without loss of generality, assume that $A_i$ is contained in $S_1 \times I,$ for all $i.$  Let $\lambda'$ be the limit, up to subsequence, of $\partial A_i$ in $PML(\partial M)$ the space of projective measured laminations.  Since $\partial A_i$ is disjoint from the meridian for all $i$, so is $\lambda'$.  In particular, $\lambda'$ lies on $S_1.$  Since $i(\lambda, \lambda')=0$ either $\lambda$ and $\lambda'$ are disjoint or $\lambda \subset \lambda'$ since $\lambda$ is minimal.  If $\lambda$ and $\lambda'$ are disjoint, then we can find a simple closed curve $c$ on $S_1$ that is disjoint from $\lambda.$  Then, $c$ is the boundary of an essential annulus $A_c$ that misses $\lambda,$ a contradiction.  If $\lambda \subset \lambda',$ then $\lambda$ lies on $S_1$ and misses any essential annulus in $S_2 \times I,$ a contradiction.
\end{proof}

\section*{Acknowledgements}
We thank the organizers of the workshop on ``Ergodic Decompositions
of Representation Varieties" held in Princeton in October 2011,
where the layout of the paper was conceived.  The second author thanks Dick Canary for several useful conversations.

\end{document}